\documentclass[11pt]{amsart}

\usepackage{times}
\usepackage{amssymb}
\usepackage{graphicx,xspace}
\usepackage{epsfig}
\usepackage{enumitem}
\usepackage[usenames,dvipsnames]{xcolor}
\usepackage{tikz}
\usepackage{gensymb}
\usepackage[T1]{fontenc}
\usepackage[utf8]{inputenc}
\usepackage{bm}
\usepackage{genyoungtabtikz}
\usepackage[config, labelfont={normalsize}]{caption,subfig}

\usepackage{hyperref}
\usepackage{todonotes}

\usetikzlibrary{matrix,arrows,calc}
\usetikzlibrary{decorations.markings}
\usetikzlibrary{patterns}
\usetikzlibrary{snakes}
\usetikzlibrary{shapes}

\usepackage[capitalize]{cleveref}

\theoremstyle{plain}

\newtheorem{lemma}{Lemma}[section]
\newtheorem{theorem}[lemma]{Theorem}
\newtheorem{corollary}[lemma]{Corollary}
\newtheorem{proposition}[lemma]{Proposition}

\newtheorem{conjecture}[lemma]{Conjecture}

\theoremstyle{remark}
\newtheorem*{remark}{Remark}

\newtheorem{definition}[lemma]{Definition}

\newcommand{\N}{\mathbb{N}}

\newcommand{\QQ}{\mathbb{Q}}

\newcommand{\ttt}{\bm{t}}
\newcommand{\xx}{\bm{x}}
\newcommand{\yy}{\bm{y}}
\newcommand{\zz}{\bm{z}}

\newcommand{\Y}{\mathcal{Y}}

\newcommand{\kumu}{\kappa}

\newcommand{\ev}{\text{even}}
\DeclareMathOperator{\Ev}{ev}
\newcommand{\odd}{\text{odd}}
\newcommand{\PPP}{\mathcal{P}}
\DeclareMathOperator{\rk}{rk}
\def\la{\lambda}
\def\ka{\kappa}
\def\a{\alpha}
\def\si{\sigma}

\def\r{r}
\DeclareMathOperator{\RHS}{RHS}

\DeclareMathOperator{\IE}{InEx}

\DeclareMathOperator{\Aut}{Aut}

\DeclareMathOperator{\hook}{hook}

\DeclareMathOperator{\Symm}{Sym}

\def\uu{\bm{u}}
\def\vv{\bm{v}}

\author[M.~Dołęga]{Maciej Dołęga}
\address{
Wydział Matematyki i Informatyki, 
Uniwersytet im.~Adama Mickiewicza, 
Collegium Mathematicum,
Umultowska 87, 
61-614 Poznań, 
Poland, \newline \indent Instytut Matematyczny,
Uniwersytet Wrocławski,  \mbox{pl.\ Grunwaldzki~2/4,} 50-384
Wrocław, Poland}
\email{maciej.dolega@amu.edu.pl}

\author[V.~Féray]{Valentin Féray}

\address{Institut für Mathematik, Universität Zürich, Winterthurerstrasse 190, 8057 Zürich, Switzerland}
 \email{valentin.feray@math.uzh.ch}
 
 \thanks{
MD is supported from {\it Agence Nationale de la Recherche}, grant ANR 12-JS02-001-01 ``Cartaplus'' and from {\it NCN}, grant UMO-2015/16/S/ST1/00420.
VF is partially supported by grant SNF-149461 ``Dual combinatorics of Jack polynomials''.}

\keywords{Jack symmetric functions; Cumulants; Laplace-Beltrami operator}

\subjclass[2010]{05E05}

\title[Cumulants of Jack symmetric functions]
{Cumulants of Jack symmetric functions and $b$-conjecture}

\begin{document}

\maketitle

\begin{abstract}
Goulden and Jackson (1996) introduced, using Jack symmetric functions,
some multivariate generating series $\psi(\bm{x}, \bm{y}, \bm{z}; t,
1+\beta)$ that might be interpreted as a continuous deformation of the
generating series of rooted hypermaps.
They made the following conjecture:
the coefficients of $\psi(\bm{x}, \bm{y}, \bm{z}; t, 1+\beta)$
in the power-sum basis
are polynomials in $\beta$ 
with nonnegative integer coefficients
(by construction, these coefficients are rational functions in $\beta$).

We prove partially this conjecture, nowadays called \emph{$b$-conjecture}, by showing that coefficients of $\psi(\bm{x}, \bm{y}, \bm{z}; t, 1+ \beta)$ are polynomials in $\beta$ with rational coefficients.
A key step of the proof is a \emph{strong factorization property} of Jack polynomials
when the Jack-deformation parameter $\a$ tends to $0$,
that may be of independent interest.
\end{abstract}

\section{Introduction}

\subsection{Jack symmetric functions}

Jack \cite{Jack1970/1971}
introduced a family of symmetric polynomials ---
which are now known as \emph{Jack polynomials} $J^{(\alpha)}_\pi$ --- indexed by a partition 
and a deformation parameter $\alpha$.
From the contemporary point of view, probably the main motivation for studying Jack polynomials 
comes from the fact that they are a special case of the celebrated \emph{Macdonald polynomials} which 
\emph{``have found applications in special function theory, representation theory, algebraic geometry, 
group theory, statistics and quantum mechanics''} \cite{GarsiaRemmel2005}. 
Indeed, some surprising features of Jack polynomials \cite{Stanley1989} have led in the past to the discovery of 
Macdonald polynomials \cite{Macdonald1995}, and Jack polynomials have been regarded as a relatively easy 
case, which later allowed the 
understanding of the more difficult case of 
Macdonald polynomials
(the series of papers \cite{LapointeVinet1995,LapointeVinet1997} illustrates this very well).
A brief overview of Macdonald polynomials and their relationship to 
Jack polynomials is given in \cite{GarsiaRemmel2005}.
Jack polynomials are also interesting on their own,
for instance in the context of Selberg integrals \cite{Kadell1997}
and in theoretical physics \cite{Feigin2002,BernevigHaldane2008}. 

Finally, according to Goulden and Jackson \cite{GouldenJackson1996},
Jack polynomials are also related to hypermap enumeration,
via specific multivariate generating functions.
This relation is still partially a conjecture,
and the main result of this paper is a step towards its resolution.

\subsection{$b$-conjecture and our main result}
\label{subsec:Introduction-BConjecture}
In the following, $\xx$, $\yy$ and $\zz$ are three disjoint infinite alphabets.
Let $J_\la^{(\a)}(\xx)$ (resp. $J_\la^{(\a)}(\yy)$, $J_\la^{(\a)}(\zz)$)
be the Jack symmetric function in $\xx$ (resp. $\yy$, $\zz$)
indexed by a partition $\la$.
Let us denote by $\hook_\a(\la)$ and $\hook'_\a(\la)$ the $\a$ hook-polynomials 
(these are combinatorial factors that appear often in Jack polynomial theory; 
see Section \ref{SubsecPartitions} for the definition).
We also use the notation $\PPP$ for the set of all integer partitions 
and $|\la|$ for the size of a partition $\la$.
Goulden and Jackson \cite{GouldenJackson1996} defined a family of coefficients
$h_{\mu,\nu}^{\tau}(\a-1)$ by the following formal series identity:
\begin{multline}
\log \left( 
\sum_{\la \in \PPP} \frac{
J_\la^{(\a)}(\xx) \, J_\la^{(\a)}(\yy) \, J_\la^{(\a)}(\zz)
\, t^{|\la|}}
{\hook_\a(\la)\, \hook'_\a(\la)}
\right) \\
 = \sum_{n \ge 1} \frac{t^n}{\a \, n} \left( 
 \sum_{\mu,\nu,\tau \vdash n} h_{\mu,\nu}^{\tau}(\a-1) \,
p_\mu(\xx) \, p_\nu(\yy) \, p_\tau(\zz) \right), 
\label{EqDefH}
\end{multline}
where $\mu,\nu,\tau \vdash n$ means that $\mu$, $\nu$ and $\tau$
are three partitions of $n$ and $p_\mu$ is the power-sum symmetric function
associated with $\mu$.

This rather involved definition is motivated by the following combinatorial
interpretations for particular values of $\a$;
see~\cite[Section 1.1]{GouldenJackson1996} and references therein.
\begin{itemize}
    \item {In the case} {$\a=1$}, the quantity $h_{\mu,\nu}^{\tau}(0)$ enumerates
        {connected hypergraphs embedded into oriented surfaces} with vertex-,
        edge- and face-degree distributions given by $\mu$, $\nu$ and 
        $\tau$, respectively.
    \item {In the case} {$\a=2$}, the quantity $h_{\mu,\nu}^{\tau}(1)$ enumerates
        {connected hypergraphs embedded into non-oriented surfaces} with the same degree conditions.
\end{itemize}
Connected hypergraphs embedded into surfaces are usually called {\em maps} and are a classical
topic in enumerative combinatorics related to the computation of matrix integrals
and to the study of moduli spaces of curves, as explained in detail in the book \cite{LandoZvonkin2004}.
The logarithm in Eq. \eqref{EqDefH} is present
because we only want to count connected objects.

Note that $h_{\mu, \nu}^\tau(\a-1)$ depends on the parameter $\a$,
and describing it as a function of $\beta := \alpha-1$ might seem be artificial.
However, this shift seems to be the right one for finding a combinatorial interpretation of $h_{\mu, \nu}^\tau(\beta)$, as suggested by Goulden and Jackson \cite{GouldenJackson1996} in the following conjecture.

\begin{conjecture}[$b$-conjecture]
\label{conj:BConj}
For all partitions $\tau, \mu, \nu \vdash n \geq 1$, the quantity $h_{\mu, \nu}^\tau(\beta)$ is a polynomial in $\beta$ with nonnegative, integer coefficients. Moreover, there exists a statistics $\eta$ on maps such that
\begin{equation}
\label{eq:StatisticOnMaps}
h_{\mu, \nu}^\tau(\beta) = \sum_{\mathcal{M}}\beta^{\eta(\mathcal{M})},
\end{equation}
where the summation index runs over all rooted, bipartite maps $\mathcal{M}$ with face
distribution $\tau$, black vertex distribution $\mu$ and white vertex distribution $\nu$, and $\eta(\mathcal{M})$ is a
nonnegative integer equals to $0$ if and only if $\mathcal{M}$ is orientable.
\end{conjecture}

This conjecture is still open.
The thesis of La Croix \cite{LaCroix2009} gives a number of evidences for it,
and gives a good account of what is known so far.
In particular, some constructions for a candidate statistics $\eta$ have been given,
establishing particular cases of the conjecture 
\cite{BrownJackson2007,LaCroix2009,KanunnikovVassilieva2014}.
However, there is not much known about the structure of $h_{\mu, \nu}^\tau(\beta)$ for arbitrary partitions $\tau, \mu, \nu \vdash n$.
Strictly from the construction they are rational functions in $\beta$
with rational coefficients.
Our main result in this paper is a proof of the following
polynomiality result for 
$h_{\mu, \nu}^\tau(\beta)$ for all partitions $\tau, \mu, \nu \vdash n \geq 1$.

\begin{theorem}
  \label{theo:Polynomiality}
  For all partitions $\tau, \mu, \nu \vdash n \geq 1$, the quantity $h_{\mu, \nu}^\tau(\beta)$ is 
  a polynomial in $\beta$ of degree $2+n-\ell(\tau) - \ell(\mu) - \ell(\nu)$ with rational coefficients.
\end{theorem}

Unfortunately, the nonnegativity and the integrality of the coefficients seem out of reach with our approach.
However, the polynomiality could be useful in the investigation of \Cref{conj:BConj}.
In particular, the first author has recently found a combinatorial
description of the top-degree part of $h_{\mu, \nu}^{(n)}(\beta)$ \cite{Dolega2016}.
\Cref{theo:Polynomiality} is one of the ingredients of the proof.

\subsection{Strong factorization of Jack polynomials}
A key step in our proof is a \emph{strong factorization property} for Jack polynomials
when $\a$ tends to zero.
To state it, let us introduce a few notations.
If $\lambda^1$ and $\lambda^2$ are partitions,
we denote by $\lambda^1 \oplus \lambda^2$ their entry-wise sum.
If $\lambda^1,\dots,\lambda^r$ are partitions
and $I$ a subset of $[r]:=\{1,\dots,r\}$,
then we denote 
\[\la^I:= \bigoplus_{i \in I} \la^i.\]
\begin{theorem}
    Let $\lambda^1,\cdots,\lambda^r$ be partitions.
    Then
    \begin{equation}
        \prod_{I \subseteq [r]} \big(J_{\la^I}^{(\a)}\big)^{(-1)^{|I|}}
        =1+O(\a^{r-1}),
        \label{EqSFJack}
    \end{equation}
    \label{ThmSFJack}
when $\a \to 0$.
\end{theorem}
Here, to give sense to the $O$ notation, we consider $J_{\la^I}^{(\a)}$ as a function of 
a real variable $\a$. 
Since all involved quantities are rational functions,
it is however also possible to think of $\a$ as a formal parameter;
see \cref{def:Osymbol} for the meaning of $O(\a^{r-1})$ in this context.

The exponent $(-1)^{|I|}$ may be a bit disturbing so let us unpack the notation
for small values of $r$.
\begin{itemize}
    \item For $r=2$, \cref{EqSFJack} writes as
        \[ \frac{J^{(\a)}_{\la^1 \oplus \la^2}}{J^{(\a)}_{\la^1}J^{(\a)}_{\la^2}}
        = 1+O(\a). \]
        Since we have rational functions,
        it is equivalent to say that, for $\a=0$,
        one has the factorization property 
        $J^{(0)}_{\la^1 \oplus \la^2}=J^{(0)}_{\la^1}J^{(0)}_{\la^2}$.
        This is indeed true and follows from an explicit expression
        for $J^{(0)}_{\la}$ given by Stanley;
        see \cite[Proposition 7.6]{Stanley1989} 
        or \cref{eq:JackA0} in this paper.
        Thus, in this case, our theorem does not give anything new.
    \item For $r=3$, \cref{EqSFJack} writes as
        \[\frac{J^{(\a)}_{\la^1 \oplus \la^2 \oplus \la^3} \, J^{(\a)}_{\la^1} \, J^{(\a)}_{\la^2}
        J^{(\a)}_{\la^3}}{J^{(\a)}_{\la^1 \oplus \la^2} J^{(\a)}_{\la^1 \oplus \la^3} J^{(\a)}_{\la^2 \oplus \la^3}}
                = 1+O(\a^2). \]
        Using the above case $r=2$, it is easily seen that the left-hand side is $1+O(\a)$.
        However, our theorem says more and asserts that it is $1+O(\a^2)$,
        which is not trivial at all.
\end{itemize}
This explains the terminology {\em strong factorization property}.\medskip

The theorem has an equivalent form that uses the notion of {\em cumulants of Jack polynomials}
--- see \cref{sec:cumulants} for comments on the terminology.
For partitions $\la^1,\dots,\la^r$, we denote
\[\ka^J(\la^1,\dots,\la^r)=
   \sum_{\substack{\pi \in \PPP([r]) } } 
   \mu(\pi, \{H\}) \prod_{B \in \pi} J_{\la^B}.
\]
Here, the sum is taken over set partitions $\pi$ of $[r]$ and
$\mu$ stands for the Möbius function of the set partition lattice;
the Reader not familiar with these concepts can have a look to \cref{SubsecSetPartitions}. 
For example
\begin{align*}
 %  \kumu(J^{(\a)}_{\lambda^1})    = & J^{(\a)}_{\lambda^1}, \\ 
   \ka^J(\la^1,\la^2) & = J^{(\a)}_{\lambda^1\oplus\lambda^2}- 
                                      J^{(\a)}_{\lambda^1}J^{(\a)}_{\lambda^2}, \\ 
   \ka^J(\la^1,\la^2,\la^3) & = J^{(\a)}_{\lambda^1\oplus\lambda^2\oplus\lambda^3}- 
                                      J^{(\a)}_{\lambda^1}J^{(\a)}_{\lambda^2\oplus\lambda^3}\\  & \qquad -
                                      J^{(\a)}_{\lambda^2}J^{(\a)}_{\lambda^1\oplus\lambda^3}-
                                      J^{(\a)}_{\lambda^3}J^{(\a)}_{\lambda^1\oplus\lambda^2} + 2J^{(\a)}_{\lambda^1}J^{(\a)}_{\lambda^2}J^{(\a)}_{\lambda^3}.
\end{align*}
We then have the following estimate for cumulants of Jack polynomials
\begin{theorem}
\label{theo:CumulantsJack}
For any partitions $\la^1, \dots, \la^r$, one has
\begin{equation}
    \label{eq:AsymptFactoJack}
    \kumu^J(\la^1,\dots,\la^r) = O(\a^{r-1}).
\end{equation}
\end{theorem}
\cref{theo:CumulantsJack} is in fact equivalent to
\cref{ThmSFJack}, as shown (in a more general setting)
by \cref{PropEquivalenceSCQF}
(we need here the fact that $J_\la$ has a non-zero limit
when $\a$ tends to $0$ \cite[Proposition 7.6]{Stanley1989};
this ensures that $J_\la=O(1)$ and $J_\la^{-1}=O(1)$).
We prove \cref{theo:CumulantsJack} in \cref{sec:proof}.
\bigskip

We noticed, using computer simulations, 
that a similar property seems to hold for Macdonald polynomials $J^{(q,t)}_\lambda$.
Unfortunately, %at the time when this paper was finished, 
we were unable to prove it and we state it here as a conjecture.
Similarly to the Jack case, we define 
\[\kumu^M(\la^1,\dots,\la^r)=
   \sum_{\substack{\pi \in \PPP([r]) } } 
   \mu(\pi, \{H\}) \prod_{B \in \pi} J^{(q,t)}_{\la^B}.\]

\begin{conjecture}
\label{conj:Macdonald}
For any partitions $\la^1, \dots, \la^r$, one has:
\begin{itemize}
    \item the strong factorization property of Macdonald polynomials
      when $q$ goes to $1$, {\it i.e.}
    \begin{equation}
        \prod_{I \subset [r]} \big(J_{\la^I}^{(q,t)}\big)^{(-1)^{|I|}}
        =1+O\bigg ( (q-1)^{r-1}\bigg),
        \label{EqSFMacdo}
    \end{equation}
when $q \to 1$;
\item the following estimates on cumulants of Macdonald polynomials
    \begin{equation}
        \kumu^M(\la^1,\dots,\la^r)=O\bigg( (q-1)^{r-1} \bigg),
        \label{EqSCMacdo}
    \end{equation}
when $q \to 1$.
\end{itemize}
\end{conjecture}

As in the Jack case, the above items are equivalent by \cref{PropEquivalenceSCQF}.
Note that the case $r=2$ of both items says that 
 \[ J^{(1,t)}_{\la^1 \oplus \la^2}=J^{(1,t)}_{\la^1}J^{(1,t)}_{\la^2},\]
 which follows from the explicit expression for $J^{(1,t)}_{\la}$
 given in \cite[Chapter VI, Remark (8.4), item (iii)]{Macdonald1995}.
Moreover, we mention that \cref{conj:Macdonald} implies \cref{theo:CumulantsJack}
as a special case by substitution $q = t^\a$ 
and taking a limit $t \to 1$ 
since one has (see \cite[Chapter VI, Eq. (10.23)]{Macdonald1995}):
\[ \lim_{t \to 1}\ (1-t)^{-|\la|} J_{\la}^{(t^\a,t)}(\bm{x}) =
J^{(\a)}_{\la}(\bm{x}).\]
{\em Note added in revision:}
After submission of the current paper,
the first author has found a
proof of \cref{conj:Macdonald};
see \cite{Dolega2016a}.

\subsection{Related problems}
We finish this section mentioning two similar problems.
First, a very similar conjecture to \cref{conj:BConj} (without logarithm in \Cref{EqDefH})
was also stated by Goulden and Jackson \cite{GouldenJackson1996}.
The series obtained in this way is conjecturally a multivariate generating function of \emph{matchings},
where the exponent of $\beta$ is some combinatorial integer-valued statistics.
The conjecture is still open,
while some special cases have been solved by Goulden and Jackson in their original article
\cite{GouldenJackson1996} and recently by Kanunnikov and Vassilieva \cite{KanunnikovVassilieva2014}.
The polynomiality result for the coefficients of this series was
proven by the authors of this paper in \cite{DolegaFeray2014}
and is significant in the current work.
Indeed, together with a simple argument
given in \cref{SubsectNoPolesOutsideZero}, it reduces the proof of \cref{theo:Polynomiality}
to checking that there is no singularity in $\alpha=0$.
\medskip

The second related problem is the investigation of \emph{Jack characters}, that is suitably normalized coefficients of the power-sum expansion of Jack polynomials.
In a series of papers \cite{Lassalle2008a, Lassalle2009} Lassalle stated
some polynomiality and positivity conjectures
suggesting that a combinatorial description of these objects might exist.
Although these conjectures are not fully resolved, it was proven by us
together with Śniady \cite{DolegaFeraySniady2014} and by Śniady \cite{Sniady2015} that in some special cases indeed, such combinatorial setup exists.
Moreover, similarly to \cref{conj:BConj}, these special cases
involve hypermaps and some statistics that ``measures their non-orientability''.
\medskip

We cannot resist to state that there must be a deep connection between all these problems, and understanding it would be of great interest.

\subsection{Organization of the paper}
We describe all necessary definitions and background in \cref{sec:preliminaries},
and in \cref{sec:cumulants} we discuss cumulants and their relation with strong factorization.
\cref{sec:proof} is devoted to the proof of the strong factorization property of Jack polynomials, 
while \cref{sec:polynomiality} presents the proof of the main result, that is the polynomiality in $b$-conjecture.

\section{Preliminaries}
\label{sec:preliminaries}

\subsection{Partitions}
\label{SubsecPartitions}
We call $\lambda := (\lambda_1, \lambda_2, \dots, \lambda_l)$ \emph{a partition} of $n$
if it is a weakly decreasing sequence of positive
integers such that $\la_1+\la_2+\cdots+\la_l = n$.
Then $n$ is called {\em the size} of $\lambda$ while $l$ is {\em its length}.
As usual, we use the notation $\la \vdash n$, or $|\la| = n$, and $\ell(\la) = l$.
We denote the set of partitions of $n$ by $\Y_n$ and we define a partial order on $\Y_n$,
called {\em dominance order}, in the following way:
\[ \lambda \leq \mu \iff \sum_{i\leq j}\lambda_i \leq \sum_{i\leq j}\mu_i \text{ for any positive integer } j.\]

For any two partitions $\la \in \Y_n$ and $\mu \in \Y_m$ we can
construct two new partitions $\la \oplus \mu, \la \cup \mu \in
\Y_{n+m}$, where $\la \oplus \mu := (\la_1+\mu_1,\la_2+\mu_2,\dots)$
and $\la \cup \mu$
is obtained by merging parts of $\la$ and $\mu$ and ordering them in a decreasing fashion.
Moreover, there exists a canonical involution on the set $\Y_n$,
which associates with a partition $\la$ its \emph{conjugate partition} $\la^t$.
By definition, the $j$-th part $\la_j^t$ of the conjugate partition
is the number of positive integers $i$ such that $\la_i \ge j$.
Notice that for any two partitions $\la, \mu$, we have 
$(\la \cup \mu)^t=\la^t \oplus \mu^t$.
A partition $\la$ is identified with some geometric object, called
\emph{Young diagram} drawn in French convention,
that can be defined as follows:
\[ \la = \{(i, j):1 \leq i \leq \la_j, 1 \leq j \leq \ell(\la) \}.\]
 For any
 box $\square := (i,j) \in \la$ from Young diagram we define its \emph{arm-length} by $a(\square) := \la_j-i$ and its \emph{leg-length} by $\ell(\square) := \la_i^t-j$ (the same definitions as in \cite[Chapter I]{Macdonald1995}), see \cref{fig:ArmLeg}.

    \begin{figure}[t]
        \[
    \begin{array}{c}
    \Yfrench
    \Yboxdim{1cm}
    \begin{tikzpicture}[scale=0.75]
        \newcommand\yg{\Yfillcolour{gray!40}}
        \newcommand\yw{\Yfillcolour{white}}
        \newcommand\tralala{(i,j)}
        \newcommand\arml{a(\square)}
        \newcommand\legl{\ell(\square)}
        \tgyoung(0cm,0cm,;;;;;;;;;;,;;_6,;|4;;;;,;:;;;,;:;;,;:;)
       % \Yfillopacity{0}
       % \Ylinethick{0.3pt}
       \Yfillopacity{1}
       \Yfillcolour{gray!40}
       \tgyoung(2cm,1cm,_6\arml)
       \tgyoung(1cm,2cm,|4\legl)
       \Yfillopacity{0}
       \Ylinethick{2pt}
       \tgyoung(1cm,1cm,\square)
       \draw[->](1.5,3.5)--(1.5,2.3);
       \draw[->](1.5,4.5)--(1.5,5.7);
       \draw[->](5.7,1.5)--(7.7,1.5);
       \draw[->](4.3,1.5)--(2.3,1.5);
       \end{tikzpicture}
    \end{array}
        \]
     \caption{Arm and leg length of boxes in Young diagrams. Above
       Young diagram corresponds to partition $\la = (10,8,6,5,4,3)$.}
     \label{fig:ArmLeg}
    \end{figure}
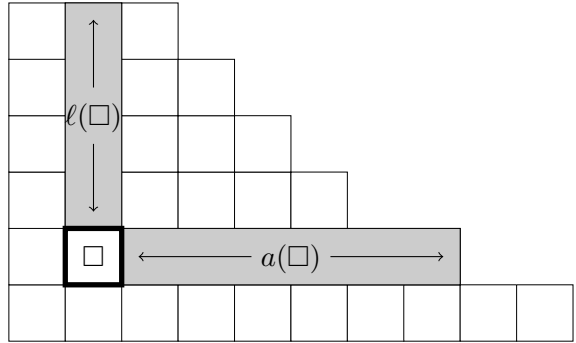

There are many combinatorial quantities associated with partitions that we will use extensively through this paper, so let us define them.
First, set
\begin{equation}
\label{eq:NumericalFactor}
z_\la := \prod_{i \geq 1}i^{m_i(\la)}m_i(\la)!,
\end{equation}
where $m_i(\la)$ denotes the number of parts of $\la$ equal to $i$.
We also define \emph{$\a$-hook polynomials} $\hook_\a(\la)$ 
and $\hook'_\a(\la)$ by the following equations:
\begin{align}
\label{eq:HookProduct}
\hook_\a(\la) &:= \prod_{\square \in \la}\left(\a\ a(\square) + \ell(\square) +1 \right),\\
\label{eq:HookProduct2}
\hook'_\a(\la) &:= \prod_{\square \in \la}\left(\a\ a(\square) + \ell(\square) +\a \right).
\end{align}
Finally, we consider a partition binomial given by
\begin{equation}
\label{eq:BinomialEigenvalue}
b(\lambda) := \sum_i \binom{\lambda_i}{2}.
\end{equation}

\subsection{Jack polynomials and Laplace-Beltrami operator}
\label{subsec:laplace}
Jack polynomials are a classical one-parameter deformation
of Schur symmetric functions, and can be defined in several different ways.
We will use a characterization via {\em Laplace--Beltrami operators}
suggested by Stanley in his seminal paper \cite[note p.~85]{Stanley1989}.
Since this is now a well-established theory,
results of this section are given without proofs but with explicit references to the literature
(mostly to Stanley's paper \cite{Stanley1989}).

First, consider the vector space $\Symm^N$ of symmetric polynomials in
$N$ variables over the field of rational functions $\QQ(\alpha)$.
The following differential operators act on this space:
\begin{equation*}
   D_1 = \sum_{i \leq N}\sum_{i \neq j}\frac{x_i^2}{x_i-x_j}\frac{\partial}{\partial x_i}, \qquad D_2 = \frac{1}{2}\sum_{i \leq N}x_i^2\frac{\partial^2}{\partial x_i^2}.
\end{equation*}
Then the \emph{Laplace--Beltrami operator} $D_\a$ 
is defined as $D_\a = D_1 + \a D_2$.
\begin{proposition}
    \label{PropDefJack}
    There exists a unique family $J_\la^{(\a)}$ (indexed by partitions $\la$ of length at most $N$)
    in $\Symm^N$ that satisfies:
    \begin{enumerate}[noitemsep,topsep=0pt,parsep=0pt,partopsep=0pt]
        \item[(C1)]
            $J^{(\a)}_\la(x_1,\dots,x_N)$ is an eigenvector of $D_\a$
with eigenvalue 
\[\Ev(\la) := \a b(\la) - b(\la^t) + (N-1)|\la|;\]
\item[(C2)] the monomial expansion of $J_\la^{(\a)}$ is given by
  \begin{equation}
    J_\la = \hook_\a(\la) m_\la + \sum_{\nu < \la}a^{\la}_\nu m_\nu,  \text{ where } a^{\la}_\nu \in \QQ(\a). 
    \label{eq:JackInMonomial}
  \end{equation}
    (Recall that we use the dominance order on partitions.)
    \end{enumerate}
    These polynomials are called {\em Jack polynomials}.
\end{proposition}
This is not the definition of Jack polynomials used by Stanley,
but the fact that Jack polynomials indeed satisfy these properties can
be found in his work \cite[Theorem 3.1 and Theorem 5.6]{Stanley1989}.
The uniqueness is an easy linear algebra exercise when one has observed that $\Ev(\la)=\Ev(\mu)$ and $|\la|=|\mu|$
imply that $\la$ and $\mu$ are either equal or incomparable for the
dominance order; see \cite[Lemma 3.2]{Stanley1989}.
A deep result of Knop and Sahi \cite{KnopSahi1997} asserts that $a^{\la}_\nu$ lies in fact in $\N[\a]$.
In particular, Jack polynomials depend polynomially on $\a$.

With the above definition, the Jack polynomial $J_\la^{(\a)}$ depends
on $N$, the number of variables.
However, it is easy to see that it satisfies the compatibility relation
$J_\la^{(\a)}(x_1,\dots,x_N,0)=J_\la^{(\a)}(x_1,\dots,x_N)$ and thus $J_\la^{(\a)}$ can be seen as a symmetric function.
In the sequel, when working with differential operators,
we sometimes confuse a symmetric function $f$
with its restriction $f(x_1,\dots,x_N,0,0,\dots)$ to $N$ variables.

Stanley also established the following specialization formula at $\a=0$:
\begin{equation}
  J^{(0)}_\lambda = \left(\prod_i \la^t_i!\right) e_{\la^t},
  \label{eq:JackA0}
\end{equation}
where $e_\la$ is the \emph{elementary symmetric function}
associated with $\la$ \cite[Proposition 7.6]{Stanley1989}.
A key point in his proof, that will be also important in the present paper,
is the following proposition.
\begin{proposition}
For any partition $\rho \vdash n$,
\label{prop:Laplace}
\begin{enumerate}[noitemsep,topsep=0pt,parsep=0pt,partopsep=0pt]
\item 
\label{prop:Laplace2}
the elementary symmetric function $e_\rho$ is an eigenvector of the operator $D_1$:
\[ D_1 e_\rho = \big((N-1)|\rho| - b(\rho)\big)e_\rho;\]
\item 
\label{prop:Laplace3}
for any partition $\mu \vdash n$ such that $b(\rho) = b(\mu)$ either $\rho = \mu$ or $\rho \nleq \mu$.
\end{enumerate}
\end{proposition}

Here is an easy corollary, that will be useful for us.
\begin{corollary}
\label{cor:EigenvectorsVanish}
Fix a partition $la$ and
let $f \in \Symm$ be a homogeneous symmetric function with an expansion in the monomial basis of the following form:
\[ f = \sum_{\mu < \lambda} d_\mu m_\mu,\]
for some $d_\mu \in \QQ$.
If, for any number $N$ of variables,
\[ D_1 f = \big((N-1)|\la| - b(\la^t)\big) f\] 
then $f = 0$.
\end{corollary}

\begin{proof}
  From the first part of \cref{prop:Laplace}, the eigenspace of the operator $D_1$ corresponding to the eigenvalue
  $\big( (N-1)|\la| - b(\la^t)\big)$ is spanned by the functions $e_{\rho^t}$ with $b_{\rho^t}=b_{\la^t}$.
  Therefore, using now the second part of \cref{prop:Laplace},
  we know that the expansion of $f$ in the elementary basis must have the following form:
  \[ f = c_\la e_{\la^t} + \sum_{\la^t \nleq \rho^t} c_\rho e_{\rho^t}.\]
  But $\la^t \nleq \rho^t$ is equivalent to $\rho \nleq \la$.
  Moreover, it is easy to see that the expansion of the elementary symmetric function $e_{\la^t}$ 
  in the monomial basis
  involves only elements $m_\mu$ indexed by partitions $\mu \leq \la$:
  \[ e_{\la^t} = m_\la + \sum_{\mu < \la} b^{\la}_\mu m_\mu.\]
  Combining these two facts we know that the expansion of $f$ in the monomial basis has the following form:
  \[ f = c_\la (m_\la + \sum_{\mu < \la} b^{\la}_\mu m_\mu) + \sum_{\rho \nleq \la} c_\rho (m_\rho + \sum_{\mu < \rho} b^{\rho}_\mu m_\mu).\]
  But we assumed that
  \[ f = \sum_{\mu < \lambda} d_\mu m_\mu,\]
  which implies that $c_\la = 0$ and $c_\rho = 0$ for all $\rho \nleq \la$, thus $f = 0$ as claimed.
\end{proof}

\subsection{Goulden and Jackson's conjectures}
\label{SubsectNoPolesOutsideZero}
Following Goulden and Jackson \cite{GouldenJackson1996} we define
\begin{align}
  \label{eq:ConnectionJackSeries}
\phi(\bm{x}, \bm{y}, \bm{z}; t, \a) &:= \sum_{n \geq 0}t^n\sum_{\lambda \vdash n} \frac{J^{(\a)}_\lambda(\bm{x})J^{(\a)}_\lambda(\bm{y})J^{(\a)}_\lambda(\bm{z})}{\left\langle J^{(\a)}_\lambda, J^{(\a)}_\lambda\right\rangle_\a}.\\
\psi(\bm{x}, \bm{y}, \bm{z}; t, \a) &:= \a t\frac{\partial}{\partial_t} \log \phi(\bm{x}, \bm{y}, \bm{z}; t, \a).
\label{eq:HypermapJackSeries}
\end{align}
We then consider their power-sum expansion, {\em i.e.}
the two families of coefficients $h_{\mu,\nu}^{\tau}$ and $c_{\mu,\nu}^{\tau}$ defined by
\begin{align}
\psi(\bm{x}, \bm{y}, \bm{z}; t, \a) &= \sum_{n \geq 1} t^n \sum_{\mu, \nu,\tau \vdash n}h_{\mu, \nu}^\tau(\a-1)p_\tau(\bm{x}) p_\mu(\bm{y}) p_\nu(\bm{z});\\
\phi(\bm{x}, \bm{y}, \bm{z}; t, \a) &= \sum_{n \geq 1} t^n \sum_{\mu, \nu,\tau \vdash n}\frac{c_{\mu, \nu}^\tau(\a-1)}{\a^{\ell(\tau)} z_\tau} p_\tau(\bm{x}) p_\mu(\bm{y}) p_\nu(\bm{z}).
\end{align}
The definition of the coefficients $h_{\mu, \nu}^\tau(\a-1)$
was already given in \cref{subsec:Introduction-BConjecture},
we recall it here to emphasize the similarity with $c_{\mu, \nu}^\tau(\a-1)$.
Goulden and Jackson conjectured that all these coefficients are polynomials in $\beta=\a -1$
with non-negative integer coefficients and some combinatorial interpretations.
The first part of their conjecture, that is the statement that
$c_{\mu, \nu}^\tau(\beta)$ are polynomials in $\beta$ with rational coefficients was recently proven by the authors of this paper:

\begin{theorem}{\cite[Proposition B.2]{DolegaFeray2014}}
For any positive integer $n$ and for any partitions $\mu,\nu,\tau \vdash n$,
the quantity $c_{\mu,\nu}^{\tau}(\beta)$ is a polynomial in $\beta$ 
(or, equivalently, in $\a$) with rational coefficients.
%of degree at most $d(\mu,\nu;\tau) := n + \ell(\tau) - \ell(\mu)-\ell(\nu)$..
\end{theorem}

Recall from \cref{eq:HypermapJackSeries}, that
\[ \psi(\bm{x}, \bm{y}, \bm{z}; t, \a) /\a =  t\frac{\partial}{\partial_t} \log \phi(\bm{x}, \bm{y}, \bm{z}; t, \a).\]
Therefore, the coefficients of the power-sum expansion of the left-hand side
--- that correspond to $h_{\mu, \nu}^\tau(\beta)/\a$ --- can be expressed as polynomials 
in terms of the coefficients of the power-sum expansion of $\phi$
--- that correspond to $|\la| c_{\mu, \nu}^\tau(\beta)/(\a^{\ell(\tau)} z_\la)$.
In particular, an immediate corollary of the above theorem is the
following one:
\begin{corollary}
    For any positive integer $n$ and for any partitions $\mu,\nu,\tau \vdash n$,
    the coefficient $h_{\mu, \nu}^\tau(\beta)$ is a rational function in $\a$
    with only possible pole at $\alpha=0$.
    \label{CorNoPolesOutsideZero}
\end{corollary}
Showing that there is in fact no pole at $\a=0$, as claimed in \cref{theo:Polynomiality},
requires a great deal of work and is the main contribution of this paper.
\subsection{Set partitions}
\label{SectPrelim}
\label{SubsecSetPartitions}
The combinatorics of set partitions is central in the theory of cumulants 
 and will be significant in this article.
 We recall here some well-known facts about them.

 Fix a ground set $S$.
A {\em set partition} of $S$ is a (non-ordered) family of non-empty disjoint
subsets of $S$ (called parts of the partition), whose union is $S$.
In the following, we always assume that $S$ is finite.

Denote by $\PPP(S)$ the set of set partitions of $S$.
Then $\PPP(S)$ may be endowed with a natural partial order:
the {\em refinement} order.
We say that $\pi$ is {\em finer} than $\pi'$ (or $\pi'$ {\em coarser} than $\pi$)
if every part of $\pi$ is included in some part of $\pi'$.
We denote this by $\pi \leq \pi'$.

Endowed with this order, $\PPP(S)$ is a complete lattice, which means that
each family $F$ of set partitions admits a join (the finest set partition
which is coarser than all set partitions in $F$; we denote the join operator by $\vee$)
and a meet (the coarsest set partition
which is finer than all set partitions in $F$; we denote the meet operator by $\wedge$).
In particular, the lattice $\PPP(S)$ has a maximum $\{S\}$ (the partition in only one
part) and a minimum $\{ \{x\}, x \in S\}$ (the partition in singletons).

Moreover, this lattice is ranked:
the rank $\rk(\pi)$ of a set partition $\pi$ is $|S|-\#(\pi)$,
where $\#(\pi)$ denotes the number of parts of $\pi$.
The rank is compatible with the lattice structure in the following sense:
for any two set partitions $\pi$ and $\pi'$,
\begin{equation}\label{EqRkJoin}
\rk(\pi \vee \pi') \leq \rk(\pi) + \rk(\pi').
\end{equation}

Lastly, denote by $\mu$ the M\"obius function of the partition lattice $\PPP(S)$.
In this paper, we only use evaluations of $\mu$ at pairs $(\pi,\{S\})$
(that is, the second argument is the one-part partition of $S$,
which is the maximum of $\PPP(S)$).
In this case, the value of the M\"obius function is given by:
\begin{equation}\label{EqValueMobius}
    \mu(\pi, \{S\})=(-1)^{\#(\pi)-1}\ (\# (\pi)-1)!.
\end{equation}

\section{Cumulants}
\label{sec:cumulants}

\subsection{Partial cumulants}
\begin{definition}    
    Let $(u_I)_{I \subseteq J}$
    be a family of elements in a field,
    indexed by subsets of a finite set $J$.
    Then its {\em partial cumulants} are defined as follows.
        For any non-empty subset $H$ of $J$, set
        \begin{equation}
            \ka_H(\uu) = \sum_{\substack{\pi \in \PPP(H) } } 
            \mu(\pi, \{H\}) \prod_{B \in \pi} u_B.
        \label{EqDefCumulants}
    \end{equation}
\end{definition}
The terminology comes from probability theory.
Let $J = [r]$, and let $X_1,\dots,X_r$ be random variables with finite moments
defined on the same probability space.
Then define $u_I=\mathbb{E}(\prod_{i \in I} X_i)$,
where $\mathbb{E}$ denotes the expected value.
The quantity $\ka_{[\r]}(\uu)$ as defined above,
is known as the {joint (or mixed) cumulant} of the random variables
$X_1,\dots,X_r$.
Also, $\ka_{H}(\uu)$ is the joint/mixed cumulant of the smaller family
$\{X_h, h \in H\}$.

Joint/mixed cumulants have been studied by Leonov and Shiryaev
in \cite{LeonovShiryaev1959} (see also an older note
of Schützenberger \cite{Schutzenberger1947}, where they are introduced
under the French name {\em déviation d'indépendence}).
They now appear 
in random graph theory \cite[Chapter 6]{JansonLuczakRucinski2000}
and have inspired a lot of work
in noncommutative probability theory;
see \cite{Novak_Sniady:What_is_free_cum} for a concise introduction to the topic.

Even if this probabilistic interpretation of cumulants is not relevant here,
we will use several lemmas that have been discovered
by the second author 
in a probabilistic context \cite{Feray2013}.

A classical result -- see, {\em e.g.},
\cite[Proposition 6.16 (vi)]{JansonLuczakRucinski2000} --
is that relation \eqref{EqDefCumulants} can be inverted as follows:
for any non-empty subset $H$ of $J$, 
\begin{equation}
   u_H = \sum_{\substack{\pi \in \PPP(H) } } 
 \prod_{B \in \pi} \ka_B(\uu).
\label{EqMomCum}
\end{equation}

\subsection{A multiplicative criterion for small cumulants}

Let $R$ be a field and $\a$ a formal parameter.
Denote by $R(\a)$ the field of rational functions in $\a$ with coefficients in $R$.
In all applications in this paper, $\a$ is the Jack parameter.

\begin{definition}
\label{def:Osymbol}
We use the following notation: for $r\in R(\a)$ and an integer $k$,
we write $r=O(\a^k)$ if the rational function 
$r \cdot \a^{-k}$ has no pole in $0$.
\end{definition}

As above, we consider a family $\uu=(u_I)_{I \subseteq [\r]}$ 
of elements of $R(\a)$ indexed by subsets of $[\r]$.
Throughout this section, we also assume that these elements are {\bf non-zero}
and $u_\emptyset=1$.

In addition to partial cumulants,
we also define the {\em cumulative factorization error terms} $T_H(\uu)$
of the family $\uu$.
The quantities $T_H(\uu)_{H \subseteq [\r],|H| \ge 2}$  are
inductively defined as follows:
for any subset $G$ of $[\r]$ of size at least $2$,
\begin{equation}
u_G = \prod_{g \in G} u_{\{g\}} \cdot  
\prod_{H \subseteq G \atop |H| \ge 2} (1 + T_H(\uu)).
\label{EqDefTInduction}
\end{equation}
Using inclusion-exclusion principle, 
a direct equivalent definition is the following one:
for any subset $H$ of $[\r]$ of size at least 2, set
\begin{equation}
T_H(\uu) = \prod_{G \subseteq H} u_G^{(-1)^{|H|}} - 1.
\label{EqDefTDirect}
\end{equation}

We have the following result.
\begin{proposition}
    \label{PropEquivalenceSCQF}
    Using the above notation, the following statements are equivalent:
    \begin{enumerate}[label=\Roman*]
        \item \hspace{-.2cm}. \label{ItemQuasiFactorization}
            {\em Strong factorization property:}
            for any subset $H \subseteq [\r]$ of size at least $2$, one has
            \begin{equation}\label{EqQuasiFactorization}
                T_{H}(\uu)
                = O\left(\a^{|H|-1}\right).
            \end{equation}
        \item \hspace{-.2cm}. \label{ItemSmallCumulants}
            {\em Small cumulant property:}
            for any subset $H \subseteq [\r]$ of size at least $2$, one has
            \begin{equation}\label{EqSmallCumulants}
                \ka_H(\uu) = \left( \prod_{h \in H} u_h \right) O\left(\a^{|H|-1}\right).
            \end{equation}
    \end{enumerate}
\end{proposition}
This proposition is a reformulation of \cite[Lemma 2.2]{Feray2013}.
However, the context and notation are quite different:
in \cite{Feray2013}, the author is interested in sequences of random variables,
while here we consider rational functions in $\a$.
Thus,
we prefer to copy the proof here, adapting it to our context.
\begin{proof}
    We first assume that $u_{\{i\}}=1$ for all $i$ in $[\r]$.

    Let us first show that \ref{ItemQuasiFactorization} implies
    \ref{ItemSmallCumulants}.
    Assume that $T_{H}(\uu)= O\left(\a^{|H|-1}\right)$,
    for any $H \subseteq [\r]$
    of size at least $2$.
    The goal is to prove that $ \kappa_{[\r]}(\uu)= O\left(\a^{\r-1}\right)$.
    This corresponds only to the case $H = [\r]$ of \ref{ItemSmallCumulants},
    but the same proof will work for any $H \subseteq [\r]$.

    Fix a set partition $\pi \in \PPP(\r)$.
    For each block $B$ of $\pi$, we expand the second product in \cref{EqDefTInduction}:
    \[ u_{B} = 
    \sum_{\{H_1,\dots,H_m\}\, \subseteq\, 2^{B}_{\geq2}} T_{H_1}(\uu) \cdots T_{H_m}(\uu),
    \]
    where $2^{B}_{\geq2}$ denotes the set of subsets of $B$
    of size at least $2$;
    the sum here runs over all (unordered) subsets 
    of $2^{B}_{\geq2}$ (in particular, the size $m$
    of this subset is not fixed).
    Therefore,
    \[\prod_{B \in \pi} u_B= \sum_{\{H_1,\dots,H_m\}\, \subseteq\, 2^{[r]}_{\geq2},\  (\star)} 
    T_{H_1}(\uu) \cdots T_{H_m}(\uu),
    \]
    where the sum runs over all sets 
    of subsets of $[\r]$ of size at least $2$
    with the following property, denoted by $(\star)$:
      each $H_i$ is contained in some block of $\pi$.
    In other terms, for each $i \in [m]$, $\pi$ must be coarser
    than the partition $\Pi(H_i)$,
    which, by definition, has $H_i$ and singletons as blocks.
    Using \cref{EqDefCumulants} and reorganizing, we get
    \begin{equation}\label{EqCumulantsSumSets}
        \kappa_{[\r]}(\uu) = \sum_{\{H_1,\dots,H_m\}\, \subseteq\, 2^{[r]}_{\geq2}} T_{H_1}(\uu) \cdots T_{H_m}(\uu) 
    \left( \sum_{\pi \in \PPP([\r]) \atop \forall i, \ \pi \geq \Pi(H_i)}
    \mu\big(\pi, \{[\r]\}\big) \right).
\end{equation}
The condition on $\pi$ can be rewritten as
\[\pi \geq \Pi(H_1) \vee \dots \vee  \Pi(H_m).\]
Hence, by definition of the M\"obius function, the sum in the parenthesis
is equal to $0$, unless $\Pi(H_1) \vee \dots \vee  \Pi(H_m) = \{[\r]\}$
(in other terms, unless the hypergraph with vertex set $[r]$ and edges $(H_i)_{1\le i \le m}$ is
connected).
On the one hand, by \cref{EqRkJoin}, it may happen only if:
\[ \sum_{i=1}^m \rk\big(\Pi(H_i)\big) = \sum_{i=1}^m (|H_i|-1) \geq \rk([\r]) = \r - 1.\]
On the other hand, one has
\[T_{H_1}(\uu) \cdots T_{H_m}(\uu) = O \left( \a^{\sum_{i=1}^m (|H_i|-1)} \right).\]
Hence only summands of order of magnitude $O(\a^k)$ for $k \ge r-1$ survive and one has
\[\kappa_{[\r]}(\uu) = O(\a^{\r-1}), \]
as wanted.
\bigskip

Let us now consider the converse statement.
We proceed by induction on $\r$ and we assume that 
for all $\r'$ smaller than a given $\r \geq 2$ 
the proposition holds.

Consider some family $(u_I)_{I \subseteq [\r]}$
such that \ref{ItemSmallCumulants} holds.
By induction hypothesis, for all $H \subsetneq [\r]$,
one has $T_H(\uu)=O(\a^{|H|-1})$.
Note that \cref{EqDefTInduction} then implies
$u_H=O(1)$ and $u_H^{-1}=O(1)$ for $H \subsetneq [\r]$.
It remains to prove that 
\[ T_{[\r]}(\uu) = \prod_{H \subseteq [\r]}       
      (u_H)^{(-1)^{|H|}} -1 =O(\a^{\r-1}).\]
Thanks to the estimates above for $u_H$, this can be rewritten as
\begin{equation}\label{EqRewriteQuasiFact}
    u_{[\r]} = \prod_{H \subsetneq [\r]} 
      (u_{H})^{(-1)^{\r-1-|H|}}
      + O(\a^{\r-1}).
  \end{equation}
  Define now an auxiliary family $\vv$:
  \[v_G = \begin{cases}
      u_G & \text{ if }G \subsetneq [\r];\\
      \prod_{H \subsetneq [\r]}          
      (u_{H})^{(-1)^{\r-1-|H|}} & \text{ for }G=[\r].
  \end{cases}\]
Clearly, since $T_G(\vv)=T_G(\uu)$ for $G \subsetneq [\r]$
and $T_{[\r]}(\vv)=0$,
the family $\vv$ has the strong factorization property.
Thus, using the first
part of the proof, it also has the small cumulant property.
In particular:
\[\kappa_{[\r]}(\vv) = O\left(\a^{\r-1}\right).\]
But, by hypothesis,
\[\kappa_{[\r]}(\uu) = O\left(\a^{\r-1}\right).\]
As $v_H=u_H$ for $H \subsetneq [\r]$, one has:
\[u_{[\r]} - v_{[\r]} =
\kappa_{[\r]}(\uu) - \kappa_{[\r]}(\vv) =
 O\left(\a^{\r-1}\right),\]
which proves \cref{EqRewriteQuasiFact}.
\bigskip

The general case follows directly from the case $u_{\{i\}}=1$
by considering the family $u'_I=u_I/\prod_{i \in I} u_{\{i\}}$.
    Indeed, for $|H| \ge 2$, it holds that
    \begin{align*}
        T_H(\uu')&=T_H(\uu);\\
        K_H(\uu')&=K_H(\uu) 
        \cdot \left(\prod_{h \in H} u_{\{h\}}\right)^{-1}. \qedhere
    \end{align*}
\end{proof}

A first consequence of this multiplicative criterion for small cumulants
is the following stability result.
\begin{corollary}
    Consider two families $(u_I)_{I \subseteq [\r]}$
    and $(v_I)_{I \subseteq [\r]}$ with the small cumulant property.
    Then their entry-wise product $(u_I v_I)_{I \subseteq [\r]}$ and quotient
    $(u_I/v_I)_{I \subseteq [\r]}$ also have the small cumulant property.
    \label{corol:stable_product}
\end{corollary}
\begin{proof}
    This is trivial for the strong factorization property and
    the small cumulant property is equivalent to it.
\end{proof}
\bigskip

\subsection{Hook cumulants}
To illustrate the above propositions and as a preparation for our next results,
we show in this section that some families involving hook polynomials 
have the small cumulant property.

We first consider $\hook_\a$ and start with a technical lemma.
\begin{lemma}
Fix a positive integer $r$ and a subset $K$ of $[r]$.
Let $(c_i)_{i \in K}$ be a family of elements of $R(\a)$, and let $C
\in R(\a)$ with $C \ne 0$.
Assume that $C$, $C^{-1}$ and all $c_i$ are $O(1)$.
For a subset $I$ of $K$ we define
\[ v_I=C + \a \cdot \sum_{i \in I} c_i.\]
Then we have, for any subset $H$ of $K$,
\[ T_H(\vv) = O(\a^{|H|}).\]
\label{LemUsefulExample}
\end{lemma}
    This is a reformulation of \cite[Lemma 2.4]{Feray2013},
    but, again, as notation is quite different there, we adapt the proof to our context.
\begin{proof}
    It is enough to prove the statement for $H=K$.
    Indeed, the case of a general set $H$ follows by considering
    the same family restricted to subsets of $H$.
    
Define $R_\ev$ (resp. $R_\odd$) as
\[\prod_\delta \left( C + \a \sum_{i \in \delta} c_i \right),\]
where the product runs over subsets of $K$ of even (resp.~odd) size.
With this notation,
$T_K(\vv)=R_\ev/R_\odd -1=(R_\ev-R_\odd)/R_\odd$.
Since $R_\odd^{-1}=O(1)$ (each term in the product is $O(1)$, as well as its inverse),
it is enough to show that $R_\ev-R_\odd=O(\a^{|K|})$.

Expanding the product in the definition of $R_\ev$, one gets
\[R_\ev = \sum_{m \geq 0} \ \ 
\frac{1}{m!}\sum_{\delta_1,\dots,\delta_m} \ \ \sum_{i_1\in \delta_1,\dots,i_m \in \delta_m}  \a^m
c_{i_1} \cdots c_{i_m} C^{2^{|K|-1} - m}.\]
The index set of the second summation symbol is the set of lists of $m$ distinct
(but not necessarily disjoint) subsets of $K$ of even size.
Of course, a similar formula with subsets of odd size holds for $R_\odd$.

Let us fix an integer $m<|K|$ and a list $i_1,\dots,i_m$.
Denote by $i_0$ the smallest integer in $K$ different from $i_1,\dots,i_m$
(as $m<|K|$, such an integer necessarily exists).
Then one has a bijection:
\[ \begin{array}{rcl}
\left\{ \begin{array}{c}
    \text{lists of subsets}\\
\delta_1,\dots,\delta_m \text{ of even size such}\\
\text{that, for all}\ h\leq m, i_h \in \delta_h
\end{array}\right\}
& \to &
\left\{ \begin{array}{c}
    \text{lists of subsets}\\
\delta_1,\dots,\delta_m \text{ of odd size such}\\
\text{that, for all}\ h\leq m, i_h \in \delta_h
\end{array}\right\}\\
(\delta_1,\dots,\delta_m)
&\mapsto&
(\delta_1 \nabla \{i_0\},\dots,\delta_m \nabla \{i_0\}),
\end{array}\]
where $\nabla$ is the symmetric difference operator.
This bijection implies that the summand $ \a^m  
c_{i_1} \cdots c_{i_m} C^{2^{|K|-2} - m}$ appears as many times
in $R_\ev$ as in $R_\odd$.
Finally, in the difference $R_\ev - R_\odd$,
terms corresponding to values of $m$ smaller than $|K|$
cancel each other and one has
\[R_\ev - R_\odd = O\big(\a^{|K|}\big). \qedhere\]
\end{proof}

We recall that given partitions $\la^1,\dots,\la^r$ and a subset $I$ of $[\r]$ we set
\[ \la^I := \bigoplus_{i \in I}\la^i,\]
where $\oplus$ is the entrywise sum; see \cref{SubsecPartitions}.

\begin{proposition}
Fix some partitions $\la^1$, \dots, $\la^r$
and for a subset $I$ of $[r]$ denote $u_I=\hook_\alpha\left(\la^I\right)$.
    The family $(u_I)$ has the strong factorization property, and hence, the small cumulant property.
    \label{PropKumuHook}
\end{proposition}
\begin{proof}
    It is enough to prove that $T_{[r]}(\uu) = O(\a^{r-1})$.

Fix some subset $I=\{i_1,\dots,i_t\}$ of $[r]$ with $i_1<\cdots<i_t$.
    Observe that the Young diagram $\la^I$ can be constructed 
    by sorting the columns of the diagrams $\la^{i_1}$, \ldots, $\la^{i_t}$ in decreasing order.
    When several columns have the same length, we put first the columns of $\la^{i_1}$,
    then those of $\la^{i_2}$ and so on; see \cref{fig:LegArmSum}
    (at the moment, please disregard symbols in boxes).
    This gives a way to identify boxes of $\la^I$ with boxes of the diagrams $\la^{i_s}$ ($1 \le s \le t$)
    that we shall use below.

    With this identification, if $b=(c,r)$ is a box in $\la^{g}$ for some $g \in I$, 
    its leg-length in $\la^I$ is the same as in $\la^{g}$.
    We denote it by $\ell(b)$.

    However, the arm-length of $b$ in $\la^I$ may be bigger 
    than the one in $\la^{g}$.
    We denote these two quantities by $a_I(b)$ and $a_{g}(b)$.
    Let us also define $a_i(b)$ for $i \neq g$ in $I$, as follows:
    \begin{itemize}
        \item for $i<g$, $a_i(b)$ is the number of boxes $b'$ in the $r$-th row of $\la^i$
            such that the size of the column of $b'$ is {\bf smaller than} the size of the column of $b$
            ({\em e.g.}, on \cref{fig:LegArmSum}, for $i=1$, these are
            boxes with a diamond);
        \item for $i>g$, $a_i(b)$ is the number of boxes $b'$ in the $r$-th row of $\la^i$
            such that the size of the column of $b'$ is {\bf at most} the size of the column of $b$
            ({\em e.g.}, on \cref{fig:LegArmSum}, for $i=3$, these are
                   boxes with an asterisk).
    \end{itemize}
    Looking at \cref{fig:LegArmSum}, it is easy to see that
    \begin{equation}
        a_I(b)= \sum_{i \in I} a_i(b).
        \label{EqArmLengthInSum}
    \end{equation}
    \begin{figure}[t]
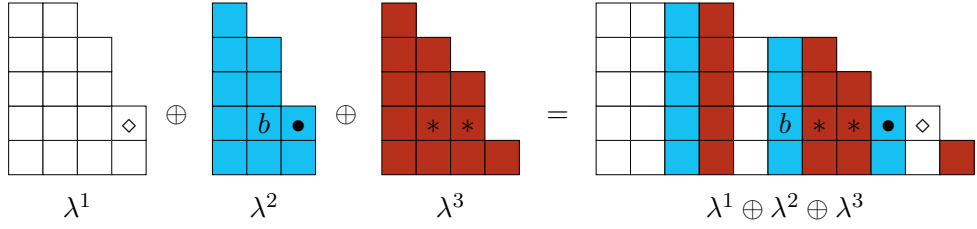

        \[
        \begin{array}{c}
        	    \YFrench
            \young(\ \ \ \ ,\ \ \ \diamond,\ \ \ ,\ \ \ ,\ \ ) \vspace{2mm} \\
        \la^1
    \end{array}\ \oplus\ \ 
    \begin{array}{c}
            \YFrench
            \Yfillcolour{cyan!70}
            \young(\ \ \ ,\ b\bullet,\ \ ,\ \ ,\ ) \vspace{2mm} \\
        \la^2
    \end{array}\ \oplus\ \ 
    \begin{array}{c}
            \YFrench
            \Yfillcolour{BrickRed}
            \young(\ \ \ \ ,\ \ast\ast,\ \ \ ,\ \ ,\ ) \vspace{2mm} \\
        \la^3
    \end{array}
    \ =\
    \begin{array}{c}
        \newcommand\yb{\Yfillcolour{cyan!70}}
        \newcommand\yr{\Yfillcolour{BrickRed}}
        \newcommand\yw{\Yfillcolour{white}}
        \YFrench
        \young(\ \ !\yb\ !\yr\ !\yw\ !\yb\ !\yr\ \ !\yb\ !\yw\ !\yr\ ,!\yw\ \ !\yb\ !\yr\ !\yw\ !\yb b!\yr\ast\ast!\yb\bullet!\yw\diamond,!\yw\ \ !\yb\ !\yr\ !\yw\ !\yb\ !\yr\ \ ,!\yw\ \ !\yb\ !\yr\ !\yw\ !\yb\ !\yr\ ,!\yw\ \ !\yb\ !\yr\ )
        \vspace{2mm} \\
        \la^1\oplus\la^2\oplus\la^3
    \end{array}\vspace{-4mm}
        \]
        \caption{The diagram of an entry-wise sum of partitions.}
        \label{fig:LegArmSum}
    \end{figure}
        Therefore, for $G \subseteq [r]$, one has:
    \[u_G=\hook_\a\left( \bigoplus_{g \in G} \la^g \right)
    = \prod_{g \in G} \left[ 
    \prod_{b \in \la^g} \ell(b)+1+ \a \ a_G(b) \right]. \] 
    From the definition of $T_{[r]}(\uu)$, given by \cref{EqDefTDirect}, we get:
    \begin{multline}
        1 + T_{[r]}(\uu) = \prod_{G \subseteq [r]} \prod_{g \in G} \left[                      
            \prod_{b \in \la^g} \ell(b)+1+ \a \, a_G(b) \right]^{(-1)^{r-|G|}} \\
            =\prod_{g \in [r]} \prod_{b \in \la^g}
            \left[ \prod_{G \subseteq [r] \atop  G \ni g} \bigg(\ell(b)+1+ \a \, a_G(b)\bigg)^{(-1)^{r-|G|}}
            \right].
            \label{EqTechnicalTHook}
    \end{multline}
    The expression inside the square bracket corresponds to $1+T_{[r]
      \setminus \{g\}}(\vv^b)$, where $(\vv^b)$ is a family indexed by
    subsets of $[r] \setminus \{g\}$ defined as follows: for a subset $I$ of $[r] \setminus \{g\}$ we set
    \[v^b_I := \ell(b)+1+ \a \ a_{I \cup \{g\}}(b).\]
    From \cref{EqArmLengthInSum}, we observe that $v^b_I$ is as in \cref{LemUsefulExample}
    with the following values of the parameters: $K=[r] \setminus \{g\}$, 
    $C=\ell(b)+1+ \a \, a_g(b)$, and $c_i = a_i(b)$ for $i \neq g$.
    Therefore we conclude that
    \[T_{[r] \setminus \{g\}}(\vv^b)= O\left(\a^{r-1}\right).\]
    Going back to \cref{EqTechnicalTHook}, we have:
    \[1 + T_{[r]}(\uu) =\prod_{g \in [r]} \prod_{b \in \la^g} \left(1 + T_{[r] \setminus \{g\}}(\vv^b)\right)=
    1+O\left(\a^{r-1}\right),\]
    which completes the proof.
\end{proof}

Let us now look at the second hook-polynomial $\hook'_\a$. 
If we try to follow the same argument as above, we want to apply \cref{LemUsefulExample}
with $K=[r] \setminus \{g\}$, $C=\ell(b)+\a (1+\, a_g(b))$, and $c_i = a_i(b)$ for $i \neq g$.
Note, however, that if the box $b$ has leg-length $0$, then $C = 0$ for $\a = 0$,
and in this case the hypothesis $C^{-1}=O(1)$ of \cref{LemUsefulExample} is not fulfilled.
To overcome this difficulty, we define
\[\hook''_\a(\la)= \prod_{\square \in \la \atop \ell(\square) \neq 0} \left(\a a(\square) + \ell(\square) +\a \right).\]
By definition, the top-most box of each column of a diagram $\la$ has leg-length $0$.
Moreover $\la$ has $m_i(\la^t)$ columns of height $i$, thus the arm-length of the
top-most boxes of these columns are $0$, $1$,\ldots, $m_i(\la^t)-1$ respectively.
Therefore,
\[\prod_{\square \in \la \atop \ell(\square) = 0} \left(\a a(\square) + \ell(\square) +\a \right)
= \a^{\la_1} \prod_i m_i(\la^t)!,\]
so that
\begin{equation}
    \hook'_\a(\la)= \a^{\la_1} \, \left( \prod_i m_i(\la^t)! \right)\, \hook''_\a(\la).
    \label{EqHPrimeHDPrime}
\end{equation}
Now, the exact same proof as for $\hook_\a$ yields the following result:
\begin{proposition}
Fix some partitions $\la^1$, \dots, $\la^r$
and set $v_I=\hook''_\alpha\left(\la^I \right)$ for all subsets $I$ of $[r]$.
    The family $(v_I)$ has the strong factorization property, and hence, the small cumulant property.
    \label{PropKumuHook2}
\end{proposition}

\section{Strong factorization property of Jack polynomials}
\label{sec:proof}

Let us fix partitions $\la^1,\dots,\la^r$, and for any subset $I
\subseteq [r]$ we define $u_I := J_{\la^I}$. The purpose of this
section is to prove \cref{theo:CumulantsJack};
with the notations above, it writes as
$\kumu^J(\la^1,\dots,\la^r) = \kumu_{[r]}(\uu) = O(\a^{r-1})$.
We start with some preliminary results.

\subsection{Preliminary results}
\label{subsec:preliminaryResults}

\begin{proposition}
\label{cor:CumulantsTopMonomialExpansion}
For any partitions $\la^1, \dots, \la^r$ the cumulant of Jack polynomials has a monomial expansion of the following form:
\[ \kumu^J(\la^1,\dots,\la^r) = \sum_{\mu < \la^{[r]}} c^{\la^1,\dots,\la^r}_\mu m_\mu + O(\a^{r-1}), \]
where the coefficients $c^{\la^1,\dots,\la^r}_\mu$ are polynomials in $\a$.
%in the monomial basis as a polynomial in $\a$ when $\a \to 0$.
\end{proposition}

\begin{proof}
First, observe that for any partitions $\nu^1$ and $\nu^2$, one has
\[ m_{\nu^1}m_{\nu^2} = m_{\nu^1 \oplus \nu^2} + \sum_{\mu < \nu^1 \oplus \nu^2}b^{\nu^1,\nu^2}_\mu m_\mu,\]
for some integers $b^{\nu^1,\nu^2}_\mu$.

Fix partitions $\la^1, \dots, \la^r$ and a set partition $\pi=\{\pi_1,\dots,\pi_s\} \in \PPP([r])$.
Note that $\la^{\pi_1} \oplus \cdots \oplus \la^{\pi_s}=\la^{[r]}$.
Thanks to \cref{eq:JackInMonomial} and the above observation on products of monomials,
there exist coefficients $d^{\la^{\pi_1},\cdots,\la^{\pi_s}}_\mu \in \QQ[\a]$ such that:
\[ J_{\la^{\pi_1}}\cdots J_{\la^{\pi_s}} = \hook_\a(\la^{\pi_1})\cdots \hook_\a(\la^{\pi_s}) m_{\la^{[r]}} + \sum_{\mu < \la^{[r]}}d^{\la^{\pi_1},\cdots,\la^{\pi_s}}_\mu m_\mu.\]
As a consequence, there exist coefficients $c^{\la^1,\dots,\la^r}_\mu \in \QQ[\a]$ such that
\[ \kumu^J(\la^1,\dots,\la^r) = \kappa_{[\r]}(\vv)m_{\la^{[r]}} + \sum_{\mu < \la^{[r]}} c^{\la^1,\dots,\la^r}_\mu m_\mu,\]
where 
$v_I=\hook_\alpha\left( \la^{I} \right)$.
\cref{PropKumuHook} completes the proof.
\end{proof}

For any positive integer $r$ and for any partitions $\la^1, \dots, \la^r$ we define
\begin{equation}
\label{eq:CumulantPartition}
\IE(\la^1,\dots,\la^r) := \sum_{I \subseteq [r]}(-1)^{r-|I|}\ b\left(\la^I\right),
\end{equation}
where $b(\la)=\sum_i \binom{\la_i}{2}$, as defined in \eqref{eq:BinomialEigenvalue}.

\begin{proposition}
\label{prop:CumulantOfPartitionsVanishes}
Let $r \geq 3$ be a positive integer. Then, for any partitions $\la^1, \dots, \la^r$ one has:
\[ \IE(\la^1,\dots, \la^r) = 0.\]
\end{proposition}

\begin{proof}
Expanding the definition and completing partitions with zeros, we have:
\[ \IE(\la^1,\dots,\la^r) = \sum_{j \ge 1} \sum_{I \subseteq [r]}(-1)^{r-|I|} \binom{\la^I_j}{2}.\]
In particular, it is enough to prove that the summand corresponding to any given $j \ge 1$ is equal to $0$.
In other terms, we can restrict ourselves to the case where $\la^i=(\la^i_1)$ has only one part.

In this case, $\IE(\la^1,\dots, \la^r)$ is a symmetric polynomial in $\la^1_1,\dots,\la^r_1$
of degree $2$ without constant term.
Moreover, its coefficients are given by:
\[ -\left[\la^1_1\right]\IE(\la^1,\dots, \la^r) = \left[\left(\la^1_1\right)^2\right]\IE(\la^1,\dots, \la^r) = \sum_{[1] \subseteq I \subseteq [r]}\frac{(-1)^{r-|I|}}{2} = 0,\] 
and 
\[ \left[\la^1_1\cdot \la^2_1\right]\IE(\la^1,\dots, \la^r) = \sum_{[2] \subseteq I \subseteq [r]}(-1)^{r-|I|} = 0.\] 
This completes the proof.
\end{proof}

Let us now define two functions that will be of great importance in the proof of \cref{theo:CumulantsJack}:
\begin{align}
\label{eq:A1}
A_1(\la^1,\dots,\la^r) &:= \sum_{\pi \in \PPP([r])} \left( \mu\big(\pi,\{[r]\}\big) \left(\sum_{B \in \pi}b\left(\la^B\right) \right)
\prod_{B \in \pi}J_{\la^B} \right),\\
\label{eq:A2}
A_2(\la^1,\dots,\la^r) &:= \sum_{\pi \in \PPP([r])} \mu\big(\pi,\{[r]\}\big)\ D_{1,2}\left(J_{\la^B}: B \in \pi\right),
\end{align}
where $D_{1,2}$ is a multivariate operator defined as follows: $D_{1,2}(f_1)=0$ and, for $k \ge 2$,
\begin{equation}
\label{eq:MixedDerivation}
D_{1,2}(f_1,\dots,f_k) := \sum_{1 \leq m \leq N}\sum_{1 \leq i < j \leq k}f_1\cdots \left(x_m\frac{\partial}{\partial x_m} f_i \right)\cdots \left(x_m\frac{\partial}{\partial x_m} f_j \right) \cdots f_k.
\end{equation}

We also recall that for any subset $I
\subseteq [r]$ we defined $u_I := J_{\la^I}$, thus
\[ \kumu_I(\uu) = \kumu^J(\la^i:i \in I).\]

\begin{lemma}
\label{lem:A1SimpleForm}
For any positive integer $r \geq 2$ and any partitions $\la^1, \dots, \la^r$, the following equality holds true:
\begin{multline*}
A_1(\la^1,\dots,\la^r) = b\left(\la^{[r]}\right)\kumu_{[r]}(\uu) + \frac{1}{2}\sum_{\emptyset \subsetneq I \subsetneq [r]} \IE\left(\la^I, \la^{I^c}\right)\kumu_I(\uu)\kumu_{I^c}(\uu),
\end{multline*}
where $I^c = [r] \setminus I$.
\end{lemma}

\begin{proof}
  We start by the following easy identity, following from \cref{EqMomCum}:
  \[\prod_{B \in \pi}J_{\la^B}= \sum_{\si \in \PPP([r]) \atop \si \le \pi} \left( \prod_{C \in \sigma} \kumu_C(\uu)\right).\]
Substituting this into the definition of $A_1$ --- \cref{eq:A1} --- we obtain
\[
A_1(\la^1,\dots,\la^r) = \sum_{\si \in \PPP([r])}\left(\sum_{\pi \in \PPP([r]); \pi \geq \si} \mu\big(\pi,\{[r]\}\big) \left(\sum_{B \in \pi}b\left(\la^B\right)\right)\right)\prod_{C \in \si}\kumu_C(\uu).
\]

Fix a set partition $\si \in \PPP([r])$. We claim that the value of
the biggest bracket is equal to
\begin{equation}
\label{eq:takietam}
\sum_{\pi \in \PPP([r]); \pi \geq \si} \mu\big(\pi,\{[r]\}\big) \left(\sum_{B \in \pi}b\left(\la^B\right)\right) = \IE\left(\la^B: B \in \si\right).
\end{equation}
Indeed, let us order the blocks of $\si$ in some way: $\si = \{B_1,\dots,B_{\#(\si)}\}$.
Partitions $\pi$ coarser than $\si$ are in bijection with partitions of the blocks of $\si$,
that is partitions of $[\#(\si)]$.
Therefore the left-hand side of \eqref{eq:takietam} can be rewritten as:
\[ \hspace{-5mm} \sum_{\pi \in \PPP([r]); \pi \geq \si} \mu\big(\pi,\{[r]\}\big) \left(\sum_{B \in \pi}b\left(\la^B\right)\right) =
\sum_{\rho \in \PPP([\#(\si)])} \mu\big(\rho,\{[\#(\si)]\}\big) \left(\sum_{C \in \rho}b\left( \bigoplus_{j \in C} \la^{B_j} \right)\right).\]
Fix some subset $C$ of $[\#(\si)]$. 
The coefficient of $b\left( \bigoplus_{j \in C} \la^{B_j} \right)$ in
the above sum is equal to
\[a_C:=\sum_{\rho \in \PPP([\#(\si)]) \atop C \in \rho} \mu\big(\rho,\{[\#(\si)]\}\big). \]
Set-partitions $\rho$ of $[\#(\si)]$ that have $C$ as a block are
uniquely expressed as $\{C\} \cup \rho'$,
where $\rho'$ is a set partition of $[\# (\si)]\setminus C$. Thus
\begin{multline*} 
a_C = \sum_{\rho' \in \PPP([\# (\si)]\setminus C)} \mu\big(\{C\} \cup \rho', \{[\# (\si)]\}\big) \\
= \sum_{0 \leq i \leq \# (\si)-|C|} S\big(\#(\si)-|C|,i\big)\ i!\ (-1)^i = (-1)^{\#(\si)-|C|}, 
\end{multline*}
where $S(n,k)$ is the \emph{Stirling number of the second kind} 
and the last equality comes from the relation
\[ \sum_{0 \leq k \leq n}S(n,k)\ (x)_k = x^n\]
evaluated at $x=-1$; here,
$(x)_k := x(x - 1)\cdots(x-k+1)$ denotes the falling factorial.
This finishes the proof of \cref{eq:takietam}.

This also completes the proof of the lemma by noticing that the right hand side of \cref{eq:takietam} vanishes for all set partitions $\si$ such that $\# (\si) \geq 3$, which is ensured by \cref{prop:CumulantOfPartitionsVanishes}.
\end{proof}

\begin{lemma}
\label{lem:A2SimpleForm}
For any positive integer $r \geq 2$ and any partitions $\la^1, \dots, \la^r$, the following equality holds true:
\begin{equation}
\label{eq:A2Identity}
A_2(\la^1,\dots,\la^r) = - \frac{1}{2}\sum_{1 \leq m \leq N}\sum_{\emptyset \subsetneq I \subsetneq [r]}\left(x_m\frac{\partial}{\partial x_m} \kumu_I(\uu) \right)\left(x_m\frac{\partial}{\partial x_m} \kumu_{I^c}(\uu) \right).
\end{equation}
\end{lemma}

\begin{proof}
    Let us call $\RHS$ the right-hand side of \cref{eq:A2Identity}.
    Using the definition of cumulants and Leibniz rule for the operator $x_m\frac{\partial}{\partial x_m}$ we get
    \[-2 \RHS =  \sum_{\emptyset \subsetneq I \subsetneq [r] \atop \pi^1 \in \PPP(I), \pi^2 \in \PPP(I^c)}
    \mu\big(\pi^1,\{I\}\big)\ \mu\big(\pi^2,\{I^c\}\big) \left( \sum_{B^1 \in \pi^1 \atop B^2 \in \pi^2}
    V_{B^1,B^2;C_1,\ldots,C_s} \right),\]
    where $C_1,\dots,C_s$ are the blocks of $\pi^1$ and $\pi^2$ distinct from $B^1$ and $B^2$
    and 
    \[V_{B^1,B^2;C_1,\ldots,C_s} = \sum_{1 \le m \le N}
    \left( x_m\frac{\partial}{\partial x_m} u_{B^1} \right) \left( x_m\frac{\partial}{\partial x_m} u_{B^2} \right)
    \left( \prod_{i=1}^s u_{C_i} \right).\]
    Fix some partition $\{B^1,B^2,C_1,\dots,C_s\}$ with two marked blocks $B^1$ and $B^2$
    (the order of two marked blocks matters)
    and consider the coefficient of $V_{B^1,B^2;C_1,\ldots,C_s}$ in $-2\RHS$.
    Pairs of set partitions $(\pi^1,\pi^2)$ contributing to this coefficient are obtained as follows:
    take a subset $J$ of $[s]$ and set
    \[\pi^1 = B^1 \cup \{C_j, j \in J \}, \ \pi^2=B^2 \cup \{C_j, j \in [s] \setminus J \}.\]
    Then $\mu(\pi^1,\{I\})=(-1)^{|J|} (|J|)!$ and $\mu(\pi^2,\{I^c\})=(-1)^{s-|J|} (s-|J|)!$.
    Thus the coefficient of $V_{B^1,B^2,C_1,\ldots,C_s}$ in $-2\RHS$
    is equal to
    \begin{multline*} \sum_{J \subset [s]} (-1)^{|J|} (|J|)! (-1)^{s-|J|} (s-|J|)!
    = \sum_{k=0}^s \binom{s}{k} (-1)^{k} k! (-1)^{s-k} (s-k)!\\
    = \sum_{k=0}^s (-1)^s s! = (-1)^s (s+1)!.
\end{multline*}
    Finally, we get
    \[\RHS =
    \frac{1}{2} \sum_{ \{B^1,B^2;C_1,\cdots,C_s\} } (-1)^{s+1} (s+1)! \, V_{B^1,B^2;C_1,\dots,C_s} ,\]
    where the sum runs over set partitions $\{B^1,B^2;C_1,\dots,C_s\}$ with two {\em ordered} marked blocks.
    Note that $(-1)^{s+1} (s+1)!$ is simply the Möbius function of the underlying set partition
    (forgetting the marked blocks) and that one can remove the factor $1/2$ by summing over
    set partitions with two {\em unordered} marked blocks.

    On the other hand, from the definition of $D_{1,2}$ --- \cref{eq:MixedDerivation} ---
    for any set partition $\pi$, one has:
    \[D_{1,2}(u_B; B \in \pi) = \sum_{\ldots} V_{B^1,B^2;C_1,\ldots,C_s},\]
    where the sum runs over all ways to mark (in an unordered way) two blocks of $\pi$;
    the resulting marked partition is then denoted $\{B^1,B^2;C_1,\dots,C_s\}$ as usual.
    Therefore, one has
    \[\sum_{\pi \in \PPP([r])} \mu\big(\pi,\{[r]\}\big) D_{1,2}(u_B; B \in \pi) = \RHS,\]
    as claimed in the lemma.
\end{proof}

\subsection{Proof of \cref{theo:CumulantsJack}}

\begin{proof}[Proof of \cref{theo:CumulantsJack}]
The proof will by given by induction on $r$.
For $r=1$, we want to prove that the Jack polynomial $J^{(\a)}_\la$ has no singularity in $\a=0$.
This follows, {\em e.g.}, from the specialization for $\a=0$ given in \cref{eq:JackA0}.
Moreover, we observed before stating \cref{theo:CumulantsJack} that the case $r=2$
also follows from \cref{eq:JackA0}.

Let us assume that the statement holds true for all $m < r$.
Notice first that, by Leibniz rule, for any $f_1,\dots,f_k \in \Symm$, one has the following expansions:
\begin{align*}
D_1 \left(f_1\cdots f_k\right) = & \sum_{1 \leq i \leq k}f_1\cdots \left(D_1 f_i \right)\cdots f_k;\\
D_2 \left(f_1\cdots f_k\right) = & \sum_{1 \leq i \leq k}f_1\cdots \left(D_2 f_i \right)\cdots f_k\\
+ & D_{1,2}(f_1,\dots,f_k),
\end{align*}
where $D_{1,2}$ is given by \cref{eq:MixedDerivation}.

Fix some partitions $\la^1$, \ldots, $\la^r$ and a set partition $\pi$ of $[r]$.
Then, one has
\begin{align*}
D_\a \left(J_{\la^{\pi_1}}\cdots J_{\la^{\pi_s}}\right) &= \sum_{1 \leq i \leq s}J_{\la^{\pi_1}}\cdots \left(D_\a J_{\la^{\pi_i}} \right)\cdots J_{\la^{\pi_s}} + \a D_{1,2} \left(J_{\la^{\pi_1}},\dots, J_{\la^{\pi_s}}\right) \\
&= \left(\sum_{1 \leq i \leq s}\bigg((N-1)|\la^{\pi_i}| - b\left((\la^{\pi_i})^t\right)\bigg)\right)J_{\la^{\pi_1}}\cdots J_{\la^{\pi_s}}\\
&+ \a \left(\left(\sum_{1 \leq i \leq s}b(\la^{\pi_i})\right)J_{\la^{\pi_1}}\cdots J_{\la^{\pi_s}} + D_{1,2} \left(J_{\la^{\pi_1}},\dots, J_{\la^{\pi_s}}\right)\right)\\
&= \bigg((N-1)\big|\la^{[r]}\big| - b\left(\left(\la^{[r]}\right)^t\right)\bigg)J_{\la^{\pi_1}}\cdots J_{\la^{\pi_s}}\\
&+ \a \left(\left(\sum_{1 \leq i \leq s}b(\la^{\pi_i})\right)J_{\la^{\pi_1}}\cdots J_{\la^{\pi_s}} + D_{1,2} \left(J_{\la^{\pi_1}},\dots, J_{\la^{\pi_s}}\right)\right),
\end{align*}
where the second equality comes from \cref{prop:Laplace}. 
Multiplying by the appropriate value of the Möbius function
and summing over set partitions $\pi$, it gives us the following identity:
\begin{multline}
\label{eq:BeltramiCumulantExpansion}
D_\a \kumu_{[r]}(\uu) = \left((N-1)\big|\la^{[r]}\big| - b\left(\left(\la^{[r]}\right)^t\right)\right)\kumu_{[r]}(\uu)\\
+ \a \left(A_1(\la^1,\dots,\la^r) + A_2(\la^1,\dots,\la^r) \right),
\end{multline}
where $A_1$ and $A_2$ are given by \cref{eq:A1} and \cref{eq:A2}, respectively.

Consider the coefficient of $\a^j$ in the above expression.
We have
\begin{multline*}
[\a^j]D_\a \kumu_{[r]}(\uu) 
= \left((N-1)\big|\la^{[r]}\big| - b\left(\left(\la^{[r]}\right)^t\right)\right)[\a^j]\kumu_{[r]}(\uu) \\
+ [\a^{j-1}]\left(A_1(\la^1,\dots,\la^r) + A_2(\la^1,\dots,\la^r) \right).
\end{multline*}
On the other hand, since $D_\a=D_1+\a \, D_2$, one has
\begin{equation}
\label{eq:InductiveHypothesisBaltrami1}
[\a^j]D_\a \kumu_{[r]}(\uu) = D_1\left([\a^j]\kumu_{[r]}(\uu) \right) + D_2\left([\a^{j-1}]\kumu_{[r]}(\uu) \right).
\end{equation}
Comparing both expressions, we have the following identity, which will be a key tool in the proof:
\begin{multline}
\label{eq:InductiveHypothesisBaltrami2}
D_1\left([\a^j]\kumu_{[r]}(\uu) \right) + D_2\left([\a^{j-1}]\kumu_{[r]}(\uu) \right) \\
= \left((N-1)\big|\la^{[r]}\big| - b\left(\left(\la^{[r]}\right)^t\right)\right)[\a^j]\kumu_{[r]}(\uu) \\
+ [\a^{j-1}]\left(A_1(\la^1,\dots,\la^r) + A_2(\la^1,\dots,\la^r) \right).
\end{multline}
\medskip

We recall that our goal is to prove that
\begin{equation}
\label{eq:InductiveHypothesis}
[\a^j]\kumu_{[r]}(\uu) = 0
\end{equation}
for any $0 \leq j \leq r-2$. We proceed by induction on $j$.

Consider the case $j=0$.
Since $\kumu_{[r]}(\uu)$, $A_1$ and $A_2$ are polynomials in $\a$,
\cref{eq:InductiveHypothesisBaltrami2} simplifies 
in this case to
\[ D_1 f = \left((N-1)\big|\la^{[r]}\big| - b\left(\left(\la^{[r]}\right)^t\right)\right) f,\]
where $f = [a^0]\kumu_{[r]}(\uu)$.
Thanks to \cref{cor:CumulantsTopMonomialExpansion} we know that $f$ satisfies 
the assumptions of \cref{cor:EigenvectorsVanish} and hence it is equal to zero.

Now, we fix $j \leq r-2$, and we assume that $[\a^i]\kumu_{[r]}(\uu) = 0$ holds true for all $ 0 \leq i < j$. We are going to show that it holds true for $i=j$ as well.

Since $[\a^{j-1}]\kumu_{[r]}(\uu) = 0$ by the inductive hypothesis, \cref{eq:InductiveHypothesisBaltrami2}  reads
\begin{multline*}
D_1\left([\a^j]\kumu_{[r]}(\uu)\right) = \left((N-1)\big|\la^{[r]}\big| - b\left(\left(\la^{[r]}\right)^t\right)\right)[\a^j]\kumu_{[r]}(\uu)\\
+ [\a^{j-1}]\left(A_1(\la^1,\dots,\la^r) + A_2(\la^1,\dots,\la^r) \right).
\end{multline*}

First, we claim that \hbox{$[\a^{j-1}]A_1(\la^1,\dots,\la^r) = 0$.}
Indeed, from the induction hypothesis, for each subset $I$ with $\emptyset \subsetneq I \subsetneq [r]$, one has $\kumu_I(\uu) = O(\a^{|I|-1})$ and \hbox{$\kumu_{I^c}(\uu) = O(\a^{|I^c|-1}) = O(\a^{r-|I|-1})$}.
We then use \cref{lem:A1SimpleForm} and write:
\begin{multline*}
[\a^{j-1}]A_1(\la^1,\dots,\la^r) = b\left(\la^{[r]}\right)[\a^{j-1}]\kumu_{[r]}(\uu) \\
+ \frac{1}{2}\sum_{\emptyset \subsetneq I \subsetneq [r]} [\a^{j-1}] \IE\left(\la^I, \la^{I^c}\right)\kumu_I(\uu)\kumu_{I^c}(\uu)= 0,
\end{multline*}
since $j-1 < r-2$. 

Similarly, one can prove that $[\a^{j-1}]A_2(\la^1,\dots,\la^r) = 0$.
Indeed, using a similar argument as before, we have 
\begin{equation*}
\frac{1}{2}\sum_{1 \leq m \leq N}\sum_{\emptyset \subsetneq I \subsetneq [r]}\left(x_m\frac{\partial}{\partial x_m} \kumu_I(\uu) \right)\left(x_m\frac{\partial}{\partial x_m} \kumu_{I^c}(\uu) \right)= O(\a^{r-2}).
\end{equation*}
But, from \cref{lem:A2SimpleForm}, the left-hand side is $A_2(\la^1,\dots,\la^r)$.
Since $j-1<r-2$, we know that $[\a^{j-1}]A_2(\la^1,\dots,\la^r) = 0$, as wanted.

Above computations show that \cref{eq:InductiveHypothesisBaltrami2} simplifies to
\[ D_1 f = \left((N-1)\left|\la^{[r]}\right| -
  b\left(\left(\la^{[r]}\right )^t\right)\right) f,\]
where $f = [a^j]\kumu_{[r]}(\uu)$.
Again, thanks to \cref{cor:CumulantsTopMonomialExpansion} we know that $f$ satisfies assumptions from \cref{cor:EigenvectorsVanish} and thus it is equal to zero, which finishes the proof.
\end{proof}

\section{Polynomiality in $b$-conjecture}
\label{sec:polynomiality}

\subsection{Cumulants and Young diagrams}
\label{SectCumYoung}
Consider a function $F$ on Young diagrams and some diagrams $\la^1,\dots,\la^r$.
Then we consider the family defined by
(recall that we use $\oplus$ for entry-wise sum of partitions):
\begin{equation}
    u_I = F\left( \bigoplus_{i \in I} \la^i \right).  
    \label{EqFamilyFLambda}
\end{equation}
\begin{definition}
    We say that a function $F$ on Young diagrams has the small cumulant property
    if, for any $r\ge 1$ and {\em for any partitions $\la^1,\dots,\la^r$},
    the above-defined family has the small cumulant property (in the
    sense of \cref{PropEquivalenceSCQF}).
\end{definition}
With this notation, the results of the previous sections can be
restated as follows:
\begin{description}
    \item[\cref{theo:CumulantsJack}]
        For a fixed alphabet $\bm{x}$, the function $\la \mapsto J_\la^\a(\bm{x})$ has the small cumulant property.
    \item[\cref{PropKumuHook}]
        The function $\hook_\a$ has the small cumulant property.
    \item[\cref{PropKumuHook2}] 
        The function $\hook''_\a$ has the small cumulant property.
    \item[\cref{corol:stable_product}] 
        If $F_1$ and $F_2$ have the small cumulant property and take non-zero values,
        then so have $F_1 \cdot F_2$ and $F_1/F_2$.
\end{description}
As a consequence, the function
\[ \la \mapsto \frac{1}{\hook_\a(\la) \hook''_\a(\a)} J_\la^\a(\bm{x}) J_\la^\a(\bm{y}) J_\la^\a(\bm{z}) \]
has the small cumulant property.
We will use that later in this section.
\bigskip

\begin{remark}
Another consequence is that the function $\la \mapsto \frac{J_\la^\a(\bm{x})}{\hook_\a(\la)}$
also has the small cumulant property.
We will not use this result here, but since
this function is the standard $P$-normalization of Jack polynomials,
we decided to mention it here.
\end{remark}

\subsection{Cumulants and logarithm}
Let $\ttt = (t_1,t_2,\dots)$ be an infinite alphabet of formal variables.
We use the notation $\ttt^\la=t_{\la_1} \cdots t_{\la_r}$.
\begin{lemma}
    Let $F$ be a function on Young diagrams.
Denote $\ka^F(\la^1,\dots,\la^r)$ the cumulant $\ka_{[r]}(\uu)$,
where $\uu$ is defined by \cref{EqFamilyFLambda}.
Then we have the following equality of formal power series in $\ttt$:
\[ \log \sum_\la \frac{F(\la)}{\a^{\la_1} \prod_i m_i(\la^t)!} \ttt^{\la^t} 
= \sum_{r \ge 1} \frac{1}{r! \alpha^r} \sum_{(j_1,\dots,j_r)} \ka^F(1^{j_1},\dots,1^{j_r})\ t_{j_1} \cdots t_{j_r}.    \]
    \label{LemLogCumulants}
\end{lemma}

\begin{proof}
    Both sides expand as linear combinations of products
    \[F_{\la^1,\ldots,\la^s} := F(\la^1) \cdots F(\la^s) \, \ttt^{(\la^1)^t} \cdots \ttt^{(\la^s)^t}, \]
    where $\la^1$, \dots, $\la^s$ are partitions.
    Fix some partitions $\la^1$, \dots, $\la^s$.
    The coefficient of $F_{\la^1,\dots,\la^s}$ on the left-hand side is given by
    \begin{equation}
        \frac{(-1)^{s-1}}{s} \, \frac{s!}{|\Aut(\la^1,\dots,\la^s)|} \,  \prod_{h=1}^s \frac{1}{\a^{\la^h_1} \prod_i m_i\left((\la^h)^t\right)!}.
        \label{EqCoefLHS}
    \end{equation}
    Here, $|\Aut(\la^1,\dots,\la^s)|$ denotes the number of permutations $\sigma$ of size $s$
    such that $\la^j=\la^{\sigma(j)}$ for all $j \le s$.
\medskip

    The situation on the right-hand side is more intricate. First, rewrite it as
    \begin{multline}
        \label{EqRHS}
        \sum_{r \ge 1} \frac{1}{r! \alpha^r} \sum_{(j_1,\dots,j_r)} \ka^F(1^{j_1},\dots,1^{j_r}) t_{j_1} \cdots t_{j_r}\\
    = \sum_{r \ge 1} \, \sum_{(j_1,\dots,j_r)} \, \sum_{\pi \in \PPP([r])} \frac{\mu(\pi,\{[r]\})}{r! \alpha^r}
    \prod_{B \in \pi} F\left(\bigoplus_{h \in B} 1^{j_h} \right) t_{j_1} \cdots t_{j_r}.
\end{multline}
    We are interested in which summation indices contribute to the coefficient of $F_{\la^1,\ldots,\la^s}$,
    that is indices such that one has the following equality of the multisets
    \[ \left\{ \bigoplus_{h \in B} 1^{j_h}, B \in \pi \right\}=\{\lambda^h, 1 \le h \le s\}.\]
    First, $(j_1,\dots,j_r)$ should be a reordering of list of column lengths in $\la^1,\dots,\la^s$.
    If $m'_i$ denotes the number of $i$ in this list of column lengths, 
    there are $r!/(\prod_i m'_i!)$ such reordering and each gives the same contribution to the coefficient of $F_{\la^1,\ldots,\la^s}$. 
    We now suppose that we have fixed such a reordering $(j_1,\dots,j_r)$.

    By definition, the number of columns of length $i$ in $\la^j$ is $m_i((\la^j)^t)$ .
    Then the number of {\em ordered} set partitions $(B_1,\dots,B_s)$ of $[r]$ such that 
    \[\bigoplus_{b \in B_h} 1^{j_b} = \lambda^h \ \text{for }1 \le h \le s\]
    is $(\prod_i m_i!)/\left(\prod_{i,j} m_i\left((\la^j)^t\right)!\right)$.
    Indeed, for each value $i$, one has to choose
    $m_i\left((\la^1)^t\right)$ entries equal to $i$ in the list $(j_1,\dots,j_r)$ that go in $B_1$,
    $m_i\left((\la^2)^t\right)$ entries equal to $i$ that go in $B_2$, and so on.
    This gives for each $i$ a multinomial $m'_i!/\left(\prod_j m_i\left((\la^j)^t\right)\right)$, as claimed.
    But we want to count (unordered) set partitions and not ordered set partitions as above,
    so that we should divide by $|\Aut(\la^1,\dots,\la^s)|$.

    All these set partitions have $s$ blocks so that the corresponding value of the Möbius function
    is $\mu(\pi,\{[r]\})=(-1)^{s-1} (s-1)!$.

    Finally, the coefficient of $F_{\la^1,\ldots,\la^s}$ in \cref{EqRHS} is
    \begin{equation}
        \frac{r!}{\prod_i m_i!} \, \frac{\prod_i m_i!}{\prod_{i,j} m_i\left((\la^j)^t\right)} \,
    \frac{1}{|\Aut(\la^1,\dots,\la^s)|} \, \frac{(-1)^{s-1} (s-1)!}{r! \alpha^r},
    \label{EqCoefRHS}
\end{equation}
    where $r$ is the total number of columns in the $\la^1,\dots,\la^s$, that is $r = \sum_h \la^h_1$.
    \medskip

    Comparing \cref{EqCoefLHS} and \cref{EqCoefRHS}, we get our result.
\end{proof}

\begin{remark}
    The statement and proof of this lemma are similar to the fact
    that cumulants can be alternatively defined as a sum over set partitions
    or as coefficients in the generating series of the logarithm of the moment generating series;
    see, {\em e.g.} \cite[Eqs (3) and (II.c)]{LeonovShiryaev1959}.
\end{remark}

\subsection{Conclusion}
We have now all the tools needed to prove the polynomiality in $b$-conjecture.
\begin{proof}[Proof of \cref{theo:Polynomiality}]
    Thanks to \cref{CorNoPolesOutsideZero}, it is enough to prove
    that $h_{\mu, \nu}^\tau(\beta)$ has no pole in $\a=0$, {\em i.e.}
    that $h_{\mu, \nu}^\tau(\beta)=O(1)$.
    From  \cref{EqDefH}, this amounts to establish that
    \[
\log \left( 
\sum_{\tau \in \PPP} \frac{J_\tau^{(\a)}(\xx) \, J_\tau^{(\a)}(\yy) \, J_\tau^{(\a)}(\zz)
\, t^{|\tau|}}{\hook_\a(\la) \hook'_\a(\la)} \right)\,
= O(\a^{-1}). \] 
But, using \cref{EqHPrimeHDPrime}, we see that this quantity is the left-hand side of
    Lemma~\ref{LemLogCumulants} for 
    \[F(\la)=\frac{1}{\hook_\a(\la) \hook''_\a(\a)} J_\la^\a(\bm{x}) J_\la^\a(\bm{y}) J_\la^\a(\bm{z}) \]
    and $t_1=t_2=\cdots=t$.
    It was observed at the end of \cref{SectCumYoung} that this function $F$
    has the small cumulant property.
    Therefore, for any $j_1,\dots,j_r$,
    the cumulant $\ka^F(1^{j_1},\dots,1^{j_r})$ is $O(\a^{r-1})$ and, thus,
    the right-hand side of \cref{LemLogCumulants} is $O(\a^{-1})$.
    This finishes the proof of the polynomiality.

    The bound on the degree follows from the polynomiality and work of La Croix,
    see \cite[Lemma 5.7 and Theorem 5.18]{LaCroix2009}.
\end{proof}

\section*{Acknowledgements}
We would like to thank P. \'Sniady for many discussions and collaborations on related topics
and an anonymous referee for pointing out a confusion between lists and sets in Sections 3.2 and 3.3.

\bibliographystyle{alpha}

\bibliography{biblio2015}

\end{document}